\documentclass[10pt]{article}
\usepackage{amsmath}
\usepackage{amsthm}
\usepackage{amsfonts}
\usepackage{amssymb}
\usepackage{latexsym} 

\usepackage[matrix,tips,graph,curve]{xy}



\makeatletter 
\@addtoreset{equation}{section}
\makeatother

\theoremstyle{plain}
\newtheorem{theorem}[equation]{Theorem}
\newtheorem{corollary}[equation]{Corollary}
\newtheorem{lemma}[equation]{Lemma}
\newtheorem{proposition}[equation]{Proposition}

\theoremstyle{definition}

\newcommand{\IC}{\mathbb{C}}

\newcommand{\IR}{\mathbb{R}}

\newcommand{\tr}{\mathrm{tr}}

\newcommand{\Tr}{\mathrm{Tr}\,}

\renewcommand{\deg}{\mathrm{deg}}

\def\d/{/\mspace{-6.0mu}/}

\newcommand{\p}{\partial}

\begin{document}

\title{Spectral asymmetry for manifolds of special holonomy}
\author{Mark Stern \footnote{Duke University, Department of Mathematics;  
 stern@math.duke.edu. The author was partially supported by NSF grant DMS-0504890.}
}

\date{\today}

\maketitle

\setcounter{section}{0}

\medskip
\section{Introduction}

Let $M^7$ be a compact $G_2$ manifold. Let $\phi$ be the defining covariant constant $3$-form. Then the $2-$forms on $M$ decompose into irreducible $G_2$ representations  
$$\Lambda^2T^*M = \Lambda_7 \oplus \Lambda_{14},$$
where $\Lambda_m$ has an $m$ dimensional fiber. Let $P_m$ denote the projection onto $\Lambda_m$. The Laplace Beltrami operator, $\Delta$, commutes with $P_m$. Denote by $\Delta_m$ the corresponding restriction of $\Delta$ to the image of $P_m$. Let $E_{m,\lambda}$ denote the $\lambda$ eigenspace of $\Delta_m$. Set 
$$N_m(x) = \sum_{0<\lambda\leq x}\text{dim}(E_{m,\lambda}).$$
As the fiber of $\Lambda_7$ is half the dimension of $\Lambda_{14}$, one expects 
$N_{14}(x)$ to grow at roughly twice the rate of $N_7(x)$.  Define 
$$\zeta_m(s) = \sum_{\lambda\in spec^+(\Delta_m)}\text{dim} (E_{m,\lambda})\lambda^{-s},$$
and set 
$$\zeta_{\delta}(s) = 2\zeta_{7}(s) - \zeta_{14}(s).$$
Here $spec^+(\Delta_m)$ denotes the nonzero spectrum of $\Delta_m$. 
As is well known (see, for example, \cite[Chapter 13]{Sh}), $\zeta_m(s)$ has a meromorphic extension to the entire complex plane with at most simple poles and is analytic for $Re(s) > \frac{7}{2}$. 
In this elementary note, we will prove the following result.
\begin{theorem}\label{th1}
The function $\zeta_{\delta}(s)$ admits an analytic extension to $Re(s) > \frac{3}{2}$. It has a simple pole at $s=\frac{3}{2}$ with residue 
$$res(\zeta_\delta)(3/2) = \frac{4}{9\pi^{2}}\int_M p_1(M)\wedge \phi,$$
where $p_1(M)$ denotes the first Pontriagin form of $TM$.  
\end{theorem} 
The integral $\int_M p_1(M)\wedge \phi$ is nonpositive and vanishes if and only if $M$ is flat. (See \cite[Proposition 10.2.6]{j}). Hence we have the following corollary. 
\begin{corollary}\label{BC1}
$$res(\zeta_\delta)(3/2) \leq 0,$$
with equality if and only if $M$ is flat.
\end{corollary} 

The proof of these results extends immediately to compact $Spin(7)$ manifolds, $X$, with defining $4$-form $\psi$.  
For $Spin(7)$ manifolds, the $2-$forms decompose into irreducible $Spin(7)$ representations 
$$\Lambda^2T^*X = \tilde\Lambda_7 \oplus \tilde\Lambda_{21}.$$
Let $\tilde\zeta_m(s)$, $m=7,21$ be the associated zeta functions of the Laplace Beltrami operator restricted to these summands. 
Let 
$$\tilde\zeta_{\delta}(s) = 3\tilde\zeta_7(s) - \tilde\zeta_{21}(s).$$
Then 

\begin{theorem}\label{th2}
The function $\tilde\zeta_{\delta}(s)$ admits an analytic extension to $Re(s) > 2$. It has a simple pole at $s=2$ with residue 
$$res(\tilde\zeta_\delta)(2) = \frac{1}{6\pi^2}\int_X p_1(X)\wedge \psi.$$ 
\end{theorem} 
For a $Spin(7)$ manifold,  $\int_M p_1(X)\wedge \psi\leq 0$ and vanishes if and only if $X$ is flat. (See \cite[Proposition 10.6.7]{j}). This gives us the following corollary. 
\begin{corollary}\label{BC2}
$$res(\tilde\zeta_\delta)(2) \leq 0,$$
with equality if and only if $X$ is flat.
\end{corollary} 

The theorems can also be extended to a twisted situation. Let $E$ be a vector bundle over a compact manifold of holonomy $G_2$ or $Spin(7)$. Suppose that $E$ is equipped with an instanton connection, $A$. Recall (see \cite{rc},\cite{dt}) that a connection on a bundle over a $G_2$ or $Spin(7)$ manifold is called an {\em instanton} if its curvature, $F_A$, satisfies 
$P_7F_A = 0.$  The connection induces an exterior derivative, 
$d_A$, on $E-$ valued forms. 
Set $\Delta_A = d_Ad_A^*+d_A^*d_A.$ The assumption that $A$ is an instanton implies 
$$0 = [\Delta_A,P_m].$$
Hence we may again decompose 
$$\Delta_A \,\,\text{restricted to 2 forms} = \Delta_{A,7}+ \Delta_{A,7j},$$
where $j=2$ for $G_2$ manifolds and $j=3$ for $Spin_7$ manifolds.
Let $\zeta_{A,m}(s),$ $\zeta_{A,\delta}(s)$, $\tilde\zeta_{A,m}(s)$, and $\tilde\zeta_{A,\delta}(s)$ be the zeta functions for the twisted Laplacians. 
Then we have 
\begin{theorem}\label{th3}
The function $\zeta_{A,\delta}(s)$ admits an analytic extension to $Re(s) > \frac{3}{2}$. It has a simple pole at $s=\frac{3}{2}$ with residue 
$$res(\zeta_{A,\delta})(3/2) = \frac{4}{3\pi^{2}}\int_M (\frac{1}{3}p_1(TM) +  c_1^2(E)-c_2(E))\wedge \phi.$$
The function $\tilde\zeta_{A,\delta}(s)$ admits an analytic extension to $Re(s) > 2$. It has a simple pole at $s=2$ with residue 
$$res(\tilde\zeta_{A,\delta})(2) = \frac{1}{2\pi^{2}}\int_X (\frac{1}{3}p_1(TM) +  c_1^2(E)-c_2(E))\wedge \psi.$$ 
\end{theorem} 
The fact that these measures of spectral asymmetry depend only on characteristic classes and the cohomology classes of $\phi$ and $\psi$ is surprising. It is easy to show that the exact forms do not contribute to these residues. Because the space of harmonic forms is finite dimensional, it does not contribute to the residues.  Hence there can be no mass cancellation of all but harmonics (as occurs in index theory) explaining the topological nature of the residues. 

 The proof of these results rapidly reduces to standard heat equation asymptotics, and requires no new techniques.  In fact, the techniques are standard for the computation of higher signatures, whose definition we now recall.
 
Let $V$ be a compact manifold. Let $f:V\rightarrow K(\pi,1)$ be a continuous map, for some group $\pi$, most often taken to be $\pi = \pi_1(V)$. Let $h\in H^p(K(\pi,1),\IR)$, and let $z$ be a de Rham representative of $f^*h.$
Let $L(TV)$ denote the Hirzebruch $L$ class of $TV$. Then 
$$\int_V z\wedge L(TV),$$
is called a {\em higher signature} of $V$. The homotopy invariance of these higher signatures is the subject of the Novikov conjecture. 
The genesis of these theorems was the observation that if we replace $z$ by the covariant constant forms $\phi$ and $\psi$ defining the special holonomies, then the associated analog of the higher signature has the spectral representation given in the preceding theorems. These computations were motivated by a desire to gain new analytic interpretations of    the higher signature invariants.  

\noindent{\bf Acknowledgements} I wish to thank Benoit Charbonneau for pointing out Corollaries \ref{BC1} and \ref{BC2} to me. 

\section{Reduction to Heat equation asymptotics}

Given a differential form $w$, denote by $e(w)$ exterior multiplication on the left by $w$. Let $\ast$ denote the Hodge star operator. 
Our computations begin in the $G_2$ case with the identification of $\Lambda_7$ and $\Lambda_{14}$ respectively  as the $+2$ and $-1$ eigenspaces of $\ast e(\phi)$. (See for example \cite{b2},\cite{j}). Similarly, in the $Spin(7)$ case, 
$\ast e(\psi)$ acts as $+3$ on $\Lambda_7$ and $-1$ on $\Lambda_{21}.$
Thus we can write 
$$\zeta_{A,\delta}(s) = \frac{1}{\Gamma(s+1)}\int_0^{\infty}\Tr\ast e(\phi)\Delta_Ae^{-t\Delta_A}t^{s}dt,$$
and 
$$\tilde\zeta_{A,\delta}(s) = \frac{1}{\Gamma(s+1)}\int_0^{\infty}\Tr\ast e(\psi)\Delta_Ae^{-t\Delta_A}t^{s}dt.$$
For any $a>0$, the integral, $\frac{1}{\Gamma(s+1)}\int_a^{\infty}\Tr\ast e(w)\Delta_Ae^{-t\Delta_A}t^{s}dt$, is bounded for all $s$. Hence the poles of the zeta functions are determined by 
$$\int_0^{1}\Tr\ast e(w)\Delta_Ae^{-t\Delta_A}\frac{t^{s}}{\Gamma(s+1)}dt
 = \int_0^{1}\Tr\ast e(w)e^{-t\Delta_A}\frac{t^{s-1}}{\Gamma(s)}dt
 - \frac{1}{\Gamma(s+1)}\Tr\ast e(w)\Delta_Ae^{-\Delta_A}.$$ The poles are therefore determined by the small $t$ asymptotics of  $\Tr\ast e(w)e^{-t\Delta_A}$. It is well known (and will be recalled in Section \ref{asympsec}) that 
with $n$ denoting the dimension of our manifold, 
$$\Tr\ast e(w)e^{-t\Delta_A} = \sum_{k=0}^{N}b_kt^{(k-n)/2} + O(t^{(N-n+1)/2}).$$
Hence, for $N+2Re(s)$ large
\begin{equation}\label{poles0}
\zeta_{A,\delta}(s) = \sum_{k=0}^{N}\frac{b_k}{\Gamma(s)}\int_0^{1}t^{(k-n)/2}t^{s-1}dt + \text{holomorphic}.\end{equation}
This gives the analytic continuation on a half plane $2Re(s)+N>n$
\begin{equation}\label{poles}
\zeta_{A,\delta}(s) = \sum_{k=0}^{N}\frac{b_k}{\Gamma(s+1)(s-(n-k)/2)} + \text{holomorphic}.\end{equation}
Hence our main theorems reduce to the well understood computations of the $b_k$. In particular, we need to show $b_k = 0$ for $k<n-\deg\,\psi $ (Proposition \ref{nores}),  and we need to compute $b_{n-\deg\psi}$ (Proposition \ref{btopres}). This type of computation is routine in many index theory contexts. A standard reference is \cite{BGV}. For the convenience of the reader unfamiliar with heat equation asymptotics, we perform the requisite computations in the following two sections.
 
\section{Algebraic trace reductions}

Let $\{\omega^i\}_{i=1}^n$ be an orthonormal coframe. 
Recall the representation of the Clifford algebra on the exterior algebra is given by defining Clifford multiplication as 
$$c(\omega^i) = e(\omega^i) - e^*(\omega^i),$$
where $e^*(w)$ is the adjoint of $e(w)$. 
Define also the operation 
$$\hat c(\omega^i) =  e(\omega^i) + e^*(\omega^i).$$
For a multi index $I$, with no repeated indices, we set 
$$c(\omega^I) = c(\omega^{i_1})c(\omega^{i_2})\cdots c(\omega^{i_n}),$$
and 
$$\hat c(\omega^I) = \hat c(\omega^{i_1})\hat c(\omega^{i_2})\cdots \hat c(\omega^{i_n}).$$
Every endomorphism of the exterior algebra can be written in the form 
$$\phi = \sum_{I,J}\phi_{IJ}c(\omega^I)\hat c(\omega^J).$$
We define the upper Clifford degree of $\phi$, 
$$deg_c^U\phi:=\max\{|I|:\phi_{IJ}\not = 0\},$$
and the lower Clifford degree
$$deg_c^L\phi:=\min\{|I|:\phi_{IJ}\not = 0\}.$$
If $deg_c^L\phi = k = deg_c^U\phi$, we say $\phi$ is homogeneous of Clifford degree $k$. 
We can expand
$$e(\omega^I) = 2^{-|I|}(c(\omega^{i_1}) + \hat c(\omega^{i_1}))\cdots (c(\omega^{i_{|I|}}) + \hat c(\omega^{i_{|I|}})),$$
and therefore see that $deg_c^Ue(\omega^I) = |I|$. 
In particular, we can write $$e(\omega^I) = 2^{-|I|}c(\omega^{I}) +\text{lower clifford degree terms}.$$ 

We recall some elementary clifford algebra trace identities.
\begin{lemma}\label{vanlem}
If $(|I|,|J|)\not = (0,0),$ then 
$$\tr \,c(\omega^I)\hat c(\omega^J) = 0.$$
\end{lemma}
\begin{proof}
$$ c(\omega^I)\hat c(\omega^J) = (-1)^{|I||J|}\hat c(\omega^J)c(\omega^I).$$
Hence cyclicity of the trace implies vanishing unless one of $|I|$ or $|J|$ is even. If $|I|>0$, 
$$ c(\omega^I)\hat c(\omega^J) = c(\omega^{i_1})c(\omega^{I\setminus \{i_1\}})\hat c(\omega^J)
=  (-1)^{(|I|-1)|J|}c(\omega^{I\setminus \{i_1\}})\hat c(\omega^J)c(\omega^{i_1}).$$
Hence cyclicity of the trace again implies vanishing if $|I|$ and $|J|$ are both even and one of the two is nonempty. Hence we are left with the case where $|I|$ and $|J|$ have different parity, and therefore one is not maximal. Assume $|I|$ is not maximal; that is assume there is $m\not\in I$. Then 
$$ c(\omega^I)\hat c(\omega^J) = -c(\omega^m)c(\omega^m)c(\omega^I)\hat c(\omega^J)
 = c(\omega^m)c(\omega^I)\hat c(\omega^J)c(\omega^m) .$$
By cyclicity of the trace, this term also vanishes. The case $|J|$ not maximal is handled similarly, completing the proof. \end{proof}

We can express the Hodge star operator in terms of clifford multiplication by the volume form. In the dimensions under consideration, this gives 
\begin{equation}\Tr\ast e(w)e^{-t\Delta_A} = -\Tr c(dvol) e(w)e^{-t\Delta_A}.\end{equation}

\section{The Asymptotics}\label{asympsec}

 We review here the construction of an approximation to $e^{-t\Delta_A}$, and then we reduce residue computations to those for harmonic oscillators as in \cite{BGV}.  
Using the Cauchy integral formula to write
$$e^{-t\Delta_A} = \frac{-1}{2\pi i}\int_{\gamma}e^{-\lambda}(t\Delta_A-\lambda)^{-1}d\lambda,$$
reduces the construction to approximating $(t\Delta_A-\lambda)^{-1}$. 
The standard method of approximation (see \cite{Gil}), which we will follow here, is to construct an approximation in coordinate neighborhoods and then patch these local approximations together using partitions of unity and auxillary cutoff functions. We suppress this latter patching step in our discussion.  

 Fix $y\in M$ and geodesic coordinates centered at $y$. Fix a frame for $E$ in a neighborhood of $y$ in which we can write 
$d_A = d+ A$ with $A(y) = 0$ and $A_{i,j}(y) = -\frac{1}{2}F_{ij}(y).$
 
 Define
$$P_{\lambda,N}f(x) = \int e^{2\pi i(x-y)\cdot u}\sum_{j=0}^N(4\pi^2t|u|^2-\lambda)^{-j-1}a_j(x,y)f(y)dydu,$$
with $a_0=Id$ in our choice of local frames. 
The remaining $a_j$ are chosen inductively with  
$$(4\pi^2t|u|^2-\lambda)^{-j}a_j(x,y) = -(t\Delta_{A,x}-4\pi itu^k\nabla_k)(4\pi^2t|u|^2-\lambda)^{-j}a_{j-1}(x,y),$$
for $1\leq j\leq N$. This gives the recipe 
$$(4\pi^2t|u|^2-\lambda)^{-j-1}a_j(x,y) = (-t)^j(4\pi^2t|u|^2-\lambda)^{-1}[(D_x^2-4\pi iu^k\nabla_k)(4\pi^2t|u|^2-\lambda)^{-1}]^ja_{0}(x,y).$$
With this choice (and continuing to suppress cutoffs and partitions of unity), 
$$(t\Delta_{A}-\lambda)P_{\lambda,N}f(x) = \int\int (t\Delta_{A,x}-\lambda)e^{2\pi i(x-y)\cdot u}\sum_{j=0}^N(4\pi^2t|u|^2-\lambda)^{-j-1}a_j(x,y)f(y)dydu$$
$$= \int\int e^{2\pi i(x-y)\cdot u}(t\Delta_{A,x}-4\pi itu^k\nabla_k + 4\pi^2|u|^2-\lambda)\sum_{j=0}^N(4\pi^2t|u|^2-\lambda)^{-j-1}a_j(x,y)f(y)dydu$$

$$= f(x) + \int\int e^{2\pi i(x-y)\cdot u}[(t\Delta_{A,x}-4\pi itu^k\nabla_k)(4\pi^2t|u|^2-\lambda)^{-N-1}a_N(x,y)]f(y)dydu,$$

Inserting this back into our expression for $e^{-t\Delta_{A}}$ gives, for a suitable curve $\gamma$ in $\IC$ surrounding the real axis, 
the approximate heat kernel 
$$p^N_t(x,y) = \int_{\gamma}\frac{-e^{-\lambda}}{2\pi i}\int e^{2\pi i(x-y)\cdot u}\sum_{j=0}^N(-t)^j(4\pi^2t|u|^2-\lambda)^{-1}[(\Delta_{A,x}-4\pi iu^k\nabla_k)(4\pi^2t|u|^2-\lambda)^{-1}]^ja_{0}(x,y)dud\lambda.$$
The error term $p_t- p_t^N$ has trace class norm which is decreasing faster than $O(t^{N/4})$, (not sharp) for $N$ large and $t\rightarrow 0$. 
Expand
$$[(\Delta_{A,x}-4\pi iu^k\nabla_k)(4\pi^2t|u|^2-\lambda)^{-1}]^ja_0 = \sum_{l,J,p}(4\pi^2t|u|^2-\lambda)^{-l}u^Jt^pa_{j,l,J,p}(x,y).$$

Inserting this into our expression for $p^N_t(x,y)$, changing the order of integration, and performing the contour integral gives 

$$p^N_t(x,y) = \int e^{-4\pi^2t|u|^2} e^{2\pi i(x-y)\cdot u}\sum_{j=0}^N(-t)^j\sum_{l,J,p}u^Jt^pa_{j,l,J,p}(x,y)du.$$
Evaluating at $x=y$ gives 

\begin{equation}\label{asymp}
p^N_t(x,x) = \int e^{-4\pi^2|u|^2} \sum_{j=0}^N(-1)^j\sum_{l,J,p}u^Jt^{j+p-n/2-|J|/2}a_{j,l,J,p}(x,x)du.\end{equation}

Inserting this expansion into equation (\ref{poles}), we see that $\zeta_{\delta}(s)$ admits an analytic extension to $Re(s) > \frac{\deg(\psi)}{2}$
\begin{equation}\label{pole2}
 \text{if }\int_M \tr\,c(dvol)e(\psi)a_{j,l,J,p}(x,x)dV = 0\text{ for 
 }j+p-n/2-|J|/2 < -\frac{\deg(\psi)}{2}.\end{equation} 

The endomorphism $c(dvol)e(\psi)$ satisfies
$$deg^L_cc(dvol)e(\psi) = n-\deg\,\psi.$$
Hence $\tr\,c(dvol)e(\psi)a_{j,l,J,p}(x,x) = 0$, unless $\deg^U_ca_{j,l,J,p}(x,x)\geq n-\deg\,\psi.$ 

So, our theorems now reduce to standard counting of Clifford degree in the construction of the $a_{j,l,J,p}(x,x)$. We recall how this is done. 
In our choice of frame, we can write in a neighborhood of the origin $y$, of our coordinate neighborhood,
$$\nabla_i = \frac{\partial}{\partial x^i} - \frac{1}{2}(x^j-y^j)(R_{ij}+F_{ij}) + O(|x-y|^2).$$
Since we will be evaluating at $y=x$, we pass to coordinates with $y=0$. Then we see that $\nabla$ has upper Clifford degree $2$ since $R_{ij}(0) = R_{ijkl}(0)e(dx^l)e^*(dx^k)$ has upper Clifford degree $2$. On the other hand, we will be evaluating at $x=y$; so, these connection terms can only contribute when they are differentiated. This suggests the following extension of our notion of Clifford degree.
We say that a differential operator 
$$\sum_{J}b_J(x)\frac{\p^{|J|}}{\p x^J}$$
has total degree (at $0$)
$$\deg_T  \sum_{J}b_J(x)\frac{\p^{|J|}}{\p x^J}  := max\{|J| -|I| + deg^U_c\frac{\partial b_J}{\p x^I}(0) : \frac{\partial b_J}{\p x^I}(0)\not = 0 \}.$$
We see that $\deg_T\nabla_i = 1$, and from Bochner's formula we obtain $\deg_T\Delta_A  = 2.$
Total degree satisfies for endomorphism valued differential operators $A$ and $B$,
$$deg_T(AB) \leq deg_TA + \deg_TB,$$ and 
$$deg_c^U(Aa_0)(x,x) \leq deg_TA.$$
Consequently, 
$$deg_c^Ua_{j,l,J,p}(x,x) \leq 2j,$$
and 
\begin{equation}\label{wind}\tr\,c(dvol)e(\psi)a_{j,l,J,p}(x,x) = 0,\text{ unless }j\geq \frac{n-\deg\,\psi}{2}.\end{equation} 
Next observe that two operations contribute to the coefficient $u^J$ of $a_{j,l,J,p}$. One is 
the $-4\pi iu^k\nabla_k$ term in the construction of $a$. The other is differentiating 
$(4\pi^2t|u|(x)^2-\lambda)^{-1}$. Note $|u|^2(x) = g^{ij}(x)u_iu_j$. Hence 
$\frac{\p}{\p x^k}(4\pi^2t|u|(x)^2-\lambda)^{-1} = -4\pi^2t\frac{\p g^{ij}}{\p x^k}u_iu_j(4\pi^2t|u|(x)^2-\lambda)^{-2}$ has total degree $\leq-1$ since $\frac{\p g^{ij}(0)}{\p x^k} = 0$. Therefore we see that 
$$deg_c^Ua_{j,l,J,p}(x,x) \leq 2j-|J|,$$
as increasing the power of $u$ requires a corresponding decrease in the total degree. This refines (\ref{wind}) to 
\begin{equation}\label{windy}\tr\,c(dvol)e(\psi)a_{j,l,J,p}(x,x) = 0,\text{ unless }j\geq \frac{n+|J|-\deg\,\psi}{2}.\end{equation} 
\begin{proposition}\label{nores}
$\tilde\zeta_{\delta}(s)$ admits an analytic extension to $Re(s) > \frac{\deg(\psi)}{2}$,
and $\zeta_{\delta}(s)$ admits an analytic extension to $Re(s) > \frac{\deg(\phi)}{2}$.
\end{proposition}
\begin{proof} Equation (\ref{windy}) implies the criterion of (\ref{pole2}) is satisfied. Replacing $\psi$ with $\phi$ gives the result for $\zeta_{\delta}(s)$.
\end{proof}
We are left to compute the residue at $s=\frac{\deg\psi}{2}.$ This is determined by the coefficient of 
$t^{-\frac{\deg\,\psi}{2}}$ in $p_t^N(x,x).$ From (\ref{asymp}), we see this coefficient is determined by 
$$\int e^{-4\pi^2|u|^2} \sum_{j=0}^N(-1)^j\sum_{l,J,p}'u^J\tr\,c(dvol)e(\psi)a_{j,l,J,p}(x,x)du,$$
where $\sum_{l,J,p}'$ denotes the sum restricted to the set where $n/2+|J|/2-j-p=\frac{\deg\psi}{2}$. The simultaneous solution of the condition $n/2+|J|/2-j-p=\frac{\deg\psi}{2}$ and the condition on $j$ in (\ref{windy}) needed for nonvanishing trace requires $p=0$. Moreover, for nonvanishing trace in the borderline case when equality is be achieved in (\ref{windy}), each term must be maximal total weight. This allows us, in the construction of $p_t^N$ to replace $|u|^2(x)$ by $|u|^2(0)$ without affecting the residue and 
$\Delta_A$ by 
$$L_A:=-(\frac{\p}{\p x^i} - \frac{x^j}{2}R_{ij}(0))^2 - e(dx^i)e^*(dx^j)(R_{ij}(0) + F_{ij}(0))$$
$$= - \frac{\p^2}{(\p x^i)^2} + R_{ij}(0)x^j\frac{\p}{\p x^i} - \sum_{j,k,i}\frac{x^jx^k}{4}R_{ij}(0)R_{ik}(0) - e(dx^i)e^*(dx^j)(R_{ij}(0) + F_{ij}(0)).$$
Moreover, the Clifford algebra commutator $[c(f_1),c(f_2)]$, of two $2-$forms $f_1$ and $f_2$ is again given by Clifford multiplication by a two form (possibly zero). Hence this commutator (unlike the commutator between differential operators and polynomials) always reduces the total degree. Hence, in computing the residue we may discard all these Clifford commutator terms. This is the same as replacing Clifford multiplication by exterior multiplication in $L_A$. Thus discarding all terms in $L_A$ of total weight less than $2$, we are left to compute the heat kernel for 
$$\hat L_A:= - \frac{\p^2}{(\p x^i)^2} + R_{ij}(0)x^j\frac{\p}{\p x^i} - \sum_{j,k,i}\frac{x^jx^k}{4}R_{ij}(0)R_{ik}(0)$$
$$ 
- \frac{1}{4} dx^i\wedge dx^j R_{ijkl}(0)\hat c(dx^l)\hat c(dx^k) - \frac{1}{4}dx^i\wedge dx^jF_{ij}(0).$$

We may now follow \cite{BGV} and use Mehler's formula for the heat kernel of $\hat L_A$. 

Let $Q$ denote the (even parity) differential form valued symmetric matrix 
$$Q_{jk} := -\sum_{i}\frac{1}{4}\hat R_{ij}(0)\hat R_{ik}(0),$$
where $\hat R_{ij}(0)$ denotes the two form 
$\frac{1}{4}R_{ijkl}dx^k\wedge dx^l$ obtained by taking the component of $R_{ij}$ of clifford degree 2 and then replacing clifford by exterior multiplication. 

Mehler's formula gives  
$$e^{-t(\Delta + Q_{jk}x^jx^k)}(x,y)$$
$$ =  det(\frac{Q^{1/2}}{2\pi\sinh(2tQ^{1/2})})^{1/2}\exp[-((\frac{Q^{1/2}}{2\tanh(2tQ^{1/2})})_{ij}(x^ix^j + y^iy^j)- (\frac{Q^{1/2}}{\sinh(2tQ^{1/2})})_{ij}x^iy^j)].$$
The right hand side is an even analytic function of $Q^{1/2}$ and hence may be defined without actually defining a square root of $Q$.  Understanding $ R_{ij}(0)x^j\frac{\p}{\p x^i} $ as an infinitesimal rotation and substituting into the preceding then gives 

$$e^{-t\hat L_A}(x,y) =  e^{- \frac{1}{2}R_{ij}(0)x^iy^j}det(\frac{Q^{1/2}}{2\pi\sinh(2tQ^{1/2})})^{1/2}\exp[-(\frac{Q^{1/2}}{2\tanh(2tQ^{1/2})})_{ij}(x-y)^i(x-y)^j]$$
$$\times \exp[\frac{t}{4}(  dx^i\wedge dx^j R_{ijkl}(0)\hat c(dx^l)\hat c(dx^k) + 2F(0))].$$

In particular, 
$$e^{-t\hat L_A}(0,0) =  det(\frac{Q^{1/2}}{2\pi\sinh(2tQ^{1/2})})^{1/2}\exp[\frac{t}{4}(  dx^i\wedge dx^j R_{ijkl}(0)\hat c(dx^l)\hat c(dx^k) + 2F(0))].$$
Hence 
$$\tr\,\ast e(\psi)p_t(x,x)dvol =  (\psi\wedge det(\frac{Q^{1/2}}{2\pi\sinh(2tQ^{1/2})})^{1/2}\exp[\frac{t}{4}(  dx^i\wedge dx^j R_{ijkl}(0)\hat c(dx^l)\hat c(dx^k) + 2F(0))])_n,$$
where $(f)_n$ denotes the component of $f$ of degree $n$. 

Following [BGV] chapter 4, we write this in our situation (with $n=7$ or $8$ and $\deg\,\psi\geq 3$) as 

\begin{equation}\label{topres} \tr\,\ast e(\psi)p_t(x,x)dvol = (\pi t)^{-deg(\psi)/2}\psi\wedge (\frac{1}{3}p_1(TM) +  c_1^2(E)-c_2(E))+ 
O(t^{(1-deg(\psi))/2}).\end{equation}
Consequently we have the following proposition. 
\begin{proposition}\label{btopres}
\begin{equation}
b_{n-\deg(\psi)} = \pi^{-deg(\psi)/2}\int_M \psi\wedge (\frac{1}{3}p_1(TM) +  c_1^2(E)-c_2(E)).
\end{equation}
\end{proposition} 

This completes our proof of Theorems \ref{th1}, \ref{th2}, and \ref{th3}.

%
 \newcommand{\etalchar}[1]{$^{#1}$}
\def\polhk#1{\setbox0=\hbox{#1}{\ooalign{\hidewidth
  \lower1.5ex\hbox{`}\hidewidth\crcr\unhbox0}}}
\providecommand{\bysame}{\leavevmode\hbox to3em{\hrulefill}\thinspace}

%
\end{document}